\documentclass{amsart}
\usepackage{amsfonts}
\usepackage[english]{babel}
\usepackage{graphicx}
\usepackage{amscd,color}
\usepackage{amsmath}
\usepackage{amssymb}
\usepackage{tikz}

\makeatletter
\@namedef{subjclassname@2010}{%
  \textup{2010} Mathematics Subject Classification}
\makeatother

\theoremstyle{plain}

\newtheorem{corollary}{\bf Corollary}

\newtheorem{lemma}{\bf Lemma}

\newtheorem{proposition}{\bf Proposition}
\newtheorem{remark}{Remark}

\newtheorem{theorem}{\bf Theorem}

\numberwithin{equation}{section}

\title[Rigidity of Serrin-type problems via integral identities]{Rigidity of Serrin-type problems via integral identities}

\author[ J. de Lima, M. S. Santos and J. S. Sindeaux]{ J. de Lima$^1$, M. S. Santos$^{\ast,1}$ and J. S. Sindeaux$^2$}

\address{$^1$ Departamento de Matem\'{a}tica, Universidade Federal da Para\'{\i}ba, 58.051-900 Jo\~{a}o Pessoa, Para\'{\i}ba, Brazil.}
\email{marcio.santos@academico.ufpb.br}
\email{jaqueline.lima@academico.ufpb.br}
\address{$^2$Universidade Regional do Cariri\\
 	63150-000 Campos Sales, Ceará, Brazil}
\email{joyce.sindeaux@urca.br}


\thanks{$^\ast$ Corresponding author.}

\begin{document}

\begin{abstract}In this short note, we deal with Serrin-type problems in Riemannian manifolds. Firstly, we provide a Soap Bubble type theorem and rigidity results. In another direction,  we obtain a rigidity result addressed to annular regions in Einstein manifolds endowed with a conformal vector field.
\end{abstract}

\maketitle

\section{Introduction}

The classical Serrin overdetermined problem \cite{serrin} asserts that a solution exists to the boundary value problem 

\begin{eqnarray}
\Delta u  &=& -f(u) \ \ \hbox{in}\quad \Omega, \quad u=0\quad \hbox{on}\quad  \partial\Omega, \label{serrinok1}\\
%u&>&  0\quad \hbox{in}\quad int(\Omega),\\
u_{\nu}&=&-c\quad \hbox{on}\quad  \partial\Omega,\label{serrinok2}    
\end{eqnarray}
%\begin{eqnarray} \label{firstSP}
 %   \left\lbrace \begin{array}{ll}
%	\Delta u   & = -f(u) \quad \textrm{in } \Omega, \\[5pt]
%	u & = 0 \quad \textrm{on } \partial\Omega, \\[5pt]
 %   u_{\nu} & = -c \quad \textrm{on } \partial\Omega,
  %  \end{array} \right. 
%\end{eqnarray}
where $\Omega \subset \mathbb{R}^n$ is a bounded domain, $f$ is a smooth real function, $\nu$ denotes the outward unit normal vector on $\partial \Omega$, and $c>0$ is a constant,  
if and only if $\Omega$ is a ball and $u$ is a radial function. Serrin provided the complete solution to this problem, presenting remarkable motivation arising from fluid dynamics. In this work, Serrin used a technique known as the Method of Moving Planes. 

%Related to \eqref{firstSP}, the classical radial solution on the Euclidean ball centering in the origin $o$ and radius $r$ denoted by $B(o,r)$, to say
%$$u(x)=\frac{R^{2}-|x|^2}{2n}$$

%In \cite{GN}, Gidas, Ni and Nirenberg, studied the Serrin problem, assuming that $\Omega$ is an Euclidean ball and just the Dirichlet condition shows that the solution $u$ is  a radial function.

In \cite{kumpraj},  S. Kumaresan and J. Prajapat have considered the Serrin problem \eqref{serrinok1}-\eqref{serrinok2} on the hemisphere and hyperbolic space, more precisely the authors show that the solution $u$ is a radial function and $\Omega$ is a geodesic ball. 

Another technique to solve the Serrin problem was presented by Weinberger in \cite{wein}, when $f=1$. Weinberger's method involves defining a $P$-function associated with the solution to the problem \eqref{serrinok1}-\eqref{serrinok2} and applying classical maximum principles and the classical Pohozaev identity for domains contained in Euclidean manifolds. This method can be extended to Riemannian manifolds with some constraint on the Ricci tensor; however, this approach is only applicable to the particular case where $f(u)=1+nku$, see \cite{FRS, ciraolo,delay, FK, FR, Roncoroni} for recent contributions about this subject.

The classical strategies to approach the Serrin problem is related to the demonstration of the celebrated Aleksandrov's Soap Bubble Theorem, which claims the only compact embedded constant mean curvature hypersurface in the Euclidean space is a sphere. 

We recall that  the proof by Ros in \cite{ros} of Alexandrov's theorem via integral inequalities, one crucial step is the \emph{Heintze-Karcher inequality} that reads as follows: given an $n-$dimensional Riemannian manifold $(M,g)$ with non-negative Ricci curvature and given a bounded domain $\Omega\subset M$ such that the mean curvature $H$ of $\partial\Omega$ is positive, then  
\begin{equation}\label{HK_Ros}
 \frac{n-1}{n}\int_{\partial\Omega}\frac{1}{H}\geq Vol(\Omega).
\end{equation}
Moreover, the equality holds in \eqref{HK_Ros} if and only if $\Omega$ is isometric to the Euclidean ball.

Motivated by this, we provided the following Heintze-Karcher inequality

\begin{theorem}\label{1}
 Let $M$ be a manifold such that  $Ric\geq (n-1)kg$. Let  $u$ be a  solution of \eqref{serrinok1}, where $f'\leq nk$.
 If the mean curvature of $\partial\Omega$ is positive, then
 $$\frac{(n-1)}{n}f^2(0)\int_{\Omega}\frac{1}{H}\geq f(0)\int_{\Omega}f(u)$$ 
 Moreover, the equality holds if and only if $u$ is a radial function and $\Omega$ is a metric ball.
 \end{theorem}
 
Motivated by Magnanini and Poggesi, \cite{poggesi}, we present the following Soap Bubble type result as follows:

\begin{theorem} [Soap Bubble Theorem]\label{2}
%\label{soapbubble}
Let $M$ be a manifold such that  $Ric\geq (n-1)kg$. Let  $u$ be a  solution of \eqref{serrinok1}, where $f'\leq nk$ and $f(0)>0,$ then
 \begin{equation*}
\int_{\partial\Omega}(H_{0}-H)(u_{\nu})^2\geq 0, 
 \end{equation*}
 where $H$ is the mean curvature of $\partial\Omega$ and $H_{0}=\frac{(n-1)f(0)}{n c}$ with $c$ a constant given by \eqref{choiceR}. In particular, if $H\geq H_{0}$ on $\partial\Omega$ then $\Omega$ is a ball (and, a fortiori, $u_{\nu}=-c$).
\end{theorem}

In another direction, we consider a P-function related to the problem \eqref{serrinok1}, given by 
$$P(x)=|\nabla u|^2(x)+\frac{2}{n}\int_0^{u(x)}f(t)dt,$$ $x\in\Omega.$ This $P$ function plays an important role in the Serrin problem, in fact, this is a subharmonic function on $\Omega$, since that $f'\leq nk,$ see Lemma \ref{P}. 

Now, let us recall that a warped product $M=I\times_{h}N$ is a product manifold $I\times N$ endowed with a metric given by $g=dt^2+h^2(t)dM$, where $dM$ is a metric on $M$. Moreover, the vector field $X=h\partial_t$ is a closed conformal vector field with conformal factor $\varphi=h'$, that is,
$$\nabla_YX=\varphi Y,$$ for all $Y\in\mathfrak{X}(M).$ Motivated by Lee and Seo \cite{LSeo},  we are able to study the following Serrin problem

\begin{equation}\label{serrinSeointro}
\left\{\begin{array}{rcl}
\Delta u  &=& -f(u)\quad \hbox{in}\quad  \Omega\\
u&=& 0\quad \hbox{on}\quad  \partial\Omega\\
u_\nu&=& -\psi(r)\quad{on}\quad \partial\Omega,\\
\end{array}\right.
\end{equation}
where $\psi(r)$ is a positive increasing function on $[0,R]$, for $R=\max\lbrace \operatorname{dist}(p,x), x\in\partial\Omega\rbrace$, $p$ is a pole of $M$ and $\Omega\subset M.$ As an application of the maximum principle on the function $P$ above defined we obtain the following extension of the Theorem 2.1 due to Lee and Seo, \cite{LSeo}.

\begin{theorem}\label{3}
Let $M=I\times_{h}N$ be a warped product such that $\varphi=h'>0$ and $Ric\geq (n-1)k$. If $u$ is a solution of \eqref{serrinSeointro}, 
such that $f'\leq nk$ and $\frac{(n-1)f(0)}{n}\frac{h(R)}{h'(R)}\leq \psi(R)$ at $r=R$, then $\Omega$ is a geodesic ball and $u$ is a radial function.
\end{theorem}

%In ..., Roncoroni has considered a suitable P-function to solve the Serrin problem \ref{first}, in the case $f(u)=-1-nku$, on space forms.
We also consider a Serrin-type problem in annular regions $\Omega=\Omega_0\setminus\Omega_1,$ 
where  $\Omega_0$ and $\Omega_1$ are bounded domains such that $\Omega_1\subset\Omega_0$, contained in manifolds endowed with a conformal vector field as follows
%\begin{equation}\label{serrinconstintro}
%\left\{\begin{array}{rcl}
%\Delta u +nku &=& -n\quad \hbox{in}\quad  %\Omega\\
%u&=& a, u_\nu=c_1  \quad  \hbox{on} \quad  \Gamma_1\\
%u&=& 0, u_\nu=c_0  \quad \hbox{on}\quad  \Gamma_0,\\
%\end{array}\right.
%\end{equation}
%where $k,c_0,c_1\in\mathbb{R}$ and $\partial\Omega=\Gamma_0\cup\Gamma_1.$
\begin{equation}\label{serrinconstintro}
    \begin{cases}
        \Delta u +nku = -n& \hbox{in}\quad  \Omega\\
u= a,\quad u_\nu=c_1  &  \hbox{on} \quad  \Gamma_1\\
u= 0,\quad u_\nu=c_0 & \hbox{on}\quad  \Gamma_0,
    \end{cases}
\end{equation}
where $k\in\mathbb{R}$,  $\partial\Omega_0=\Gamma_0,$ $\partial\Omega_1=\Gamma_1$ and $\partial\Omega=\Gamma_1\cup\Gamma_0.$

%$k,c_0,c_1\in\mathbb{R}$ and %$\partial\Omega=\Gamma_0\cup\Gamma_1.$ 
In the Euclidean manifold, Reichel \cite{Reichel} shows that $\Omega$ is a standard annulus since that $0\leq u\leq a.$ The main tool used in the proof is the moving planes method. After, Sirakov \cite{sirakov} removed the extra condition on the solution $u.$ Recently, Lee and Seo \cite{LSeo}, provided some rigidity results in space forms by using the P-function method.

Now, we are in a position to state our contribution addressed to annular regions in Einstein manifolds endowed with a conformal vector field as follows.

\begin{theorem}\label{4}
%\label{main}
Let $M$ be an Einstein manifold with $Ric=(n-1)kg$ and endowed with a closed conformal vector field $X.$ Suppose that $u$ is a solution of the problem \eqref{serrinconstintro} such that $c_1 H_1\geq -(ka+1)(n-1)$ on $\Gamma_1$ and $c_0 H_0\leq -(n-1)$ on $\Gamma_0$. If $\partial\Omega$ is star-shaped manifold, then $\partial\Omega$ is an umbilical hypersurface.
\end{theorem}

\textbf{Overview of the Paper} In Section 2, we provide a proof of the Heintze-Karcher inequality and the Soap Bubble-type result. In section 3, we provide a rigidity result, by using a new P-function on manifolds, see Theorem \ref{torsion}, and we give a proof of Theorem \ref{3}. In Section 4, the focus is the study of the Serrin problem in annular domains contained in Einstein manifolds.

\section{Soap Bubble type theorem}
Let us consider the following Serrin-type problem on bounded domains $\Omega$ contained in a Riemannian manifold $M$ as follows.
\begin{equation}\label{serrinnoover}
\left\{\begin{array}{rcl}
\Delta u  &=& -f(u)\quad \hbox{in}\quad  \Omega\\
u&=& 0\quad \hbox{on}\quad  \partial\Omega,\\
\end{array}\right.
\end{equation}
where $f:M\rightarrow\mathbb{R}$ is a smooth function. We begin this section remembering the well-known \emph{Reilly identity}

$$\int_{\Omega}^{} \Big[ \frac{n-1}{n}(\Delta \rho)^2 - | \mathring{\nabla}^2\rho |^2 \Big] \; \label{reilly}\\      = \int_{\partial \Omega}^{} \Big( h(\bar{\nabla} z, \bar{\nabla} z) + 2 \rho_{\nu} \bar{\Delta} z + H \rho_{\nu}^2\Big)  + \int_{\Omega}^{} Ric(\nabla \rho, \nabla \rho) ,
 $$   
which holds true for every domain $\Omega$ in a Riemannian manifold $(M^n,g)$ and for every $\rho \in C^{\infty}(\overline{\Omega})$, where $\mathring{\nabla}^2 \rho$ denotes the traceless Hessian of $\rho$, explicitly
$$
\mathring{\nabla}^2 \rho=\nabla^2 \rho - \frac{\Delta \rho}{n} g\, ,
$$
$\bar{\nabla}$ and $\bar{\Delta}$ indicate the gradient and the Laplacian of the induced metric in $\partial\Omega$, $z=\rho\vert_{\partial\Omega}$ 
and $\nu$ be the unit outward normal of $\partial\Omega$, $h(X,Y)= g(\nabla_{X}\nu, Y)$ and $H=tr_{g}h$ the second fundamental form and the mean curvature (with respect to $\nu$) of $\partial\Omega$, respectively.

 Now, we can state and prove our first integral inequality as follows.

 \begin{proposition}\label{keylemma0}
 Let  $u$ be a  solution of \eqref{serrinnoover} such that $f'\leq nk$, then
 \begin{equation}\label{serrin3}  \int_{\Omega}{| \mathring{\nabla}^2 u}|^2+\int_{\Omega}[\mbox{Ric}-(n-1)kg](\nabla u,\nabla u)\leq-\frac{1}{n}\int_{\partial\Omega}u_{\nu}[(n-1)f(0)+nHu_{\nu}].
\end{equation}
The equality holds if and only if $f(u)=f(0)+nku$.

 \end{proposition}
 \begin{proof}
 Reilly's identity applied to the solution of \eqref{serrinnoover} implies
\begin{equation*}
%\label{eq1}
 \int_{\Omega}|\mathring{\nabla}^2 u|^2=-\int_{\partial\Omega}Hu_{\nu}^2-\int_{\Omega}\mbox{Ric}(\nabla u,\nabla u)+\frac{n-1}{n}\int_{\Omega}(\Delta u)^2.   
\end{equation*}
 
Note that 
$$(\Delta u)^2=-f(u)\operatorname{div}(\nabla u)=-\operatorname{div}(f(u)\nabla u)+f'(u)|\nabla u|^2.$$

Thus,
$$\int_{\Omega}(\Delta u)^2=-\int_{\partial\Omega}f(0)u_\nu+\int_{\Omega}f'(u)|\nabla u|^2.$$

Thus, from \eqref{serrinnoover} and above equation we get that
$$\int_{\Omega}|\mathring{\nabla}^2 u|^2=-\int_{\partial\Omega}Hu_{\nu}^2-\int_{\Omega}\mbox{Ric}(\nabla u,\nabla u)-\frac{n-1}{n}\int_{\partial\Omega}f(0)u_\nu+\frac{n-1}{n}\int_{\Omega}f'(u)|\nabla u|^2.$$
Taking into account that $f'\leq nk$ we get the desired inequality.

Now, supposing that the equality holds we have that $f'(u)|\nabla u|^2=nk|\nabla u|^2$ and, therefore,
$$|\nabla(-f(u)+nku)|^2=(-f'(u)+nk)^2|\nabla u|^2=0.$$
Since $u=0$ along the boundary, we have that $f(u)=f(0)+nku.$ %On another hand, since $|\mathring{\nabla}^2 u|^2=0$, we can guarantee that $\nabla^2u=-(\frac{f(0)}{n}+ku)g.$ Finally from \cite{FR} we can conclude that $u$ is radial and $\Omega$ is a metric ball.

\end{proof}

As a consequence of the above result, we get the following.

\begin{corollary}
Let $M$ be a manifold such that $Ric\geq (n-1)kg$. Let $u$ be a solution of \eqref{serrinnoover}, where $f'\leq nk$. If $u_\nu=-\frac{(n-1)f(0)}{nH}$, then $u$ is radial and $\Omega$ is a metric ball.
\end{corollary}
\begin{proof}
    Indeed, from our hypothesis and Proposition \ref{keylemma0} we get that $|\mathring{\nabla}^2 u|^2=0.$
    Since $f(u)=-f(0)-nku$  we can guarantee that $\nabla^2u=-(\frac{f(0)}{n}+ku)g.$ Finally from a slightly modification of Obata-type theorem in \cite[Lemma 6] {FR} , we can conclude that $u$ is radial and $\Omega$ is a metric ball.
\end{proof}

 \begin{proof}[Proof of Theorem \ref{1}]
 
 %\begin{theorem}\label{HK}
 %Let $M$ be a manifold such that $Ric\geq (n-1)kg$. Let $u$ be a solution of \eqref{serrinnoover}, where $f'\leq nk$.
 %If the mean curvature of $\partial\Omega$ is positive, then
 %$$\frac{(n-1)}{n}f^2(0)\int_{\Omega}\frac{1}%{H}\geq f(0)\int_{\Omega}f(u)$$ 
 %Moreover, the equality holds if and only if $u$ is a radial function and $\Omega$ is a metric ball.
 %\end{theorem}
 %\begin{proof}
 Since
 \begin{eqnarray*}
     \frac{1}{nH}[(n-1)f(0)+nHu_{\nu}]^2&=&\frac{(n-1)^2f^2(0)}{nH}+u_{\nu}[(n-1)f(0)+nHu_{\nu}]+u_{\nu}(n-1)f(0),
 \end{eqnarray*}   
we have
\begin{eqnarray*}
-\frac{1}{n}\int_{\partial \Omega}u_{\nu}[(n-1)f(0)+nHu_{\nu}]&=&-\int_{\partial\Omega}\frac{1}{n^2 H}[(n-1)f(0)+nHu_{\nu}]^2\\
& &+\left(\frac{(n-1)f(0)}{n}\right)^2\int_{\partial\Omega}\frac{1}{H}+\frac{(n-1)f(0)}{n}\int_{\partial\Omega}u_{\nu}.    
\end{eqnarray*}
On the other hand,
%\begin{eqnarray*}
$$\int_{\partial\Omega}u_{\nu}=-\int_{\Omega}f(u).$$
 %&=&\int_{\Omega}(-1-nku)\\
% &\geq&-f(0)Vol(\Omega)-nk\int_{\Omega}u.
%\end{eqnarray*}
Then, from above equalities and by \eqref{serrin3} we conclude the desired result. If the equality holds, then  $\mathring{\nabla}^2 u=0$ and, therefore, from \cite{FR}, $\Omega$ is a metric ball and $u$ is a radial function.
%\begin{eqnarray}
%& &
 %\int_{\Omega}|\mathring{\nabla}^2 u|^2+\int_{\Omega}[\mbox{Ric}-(n-1)kg](\nabla u,\nabla u)+\frac{1}{n^2}\int_{\partial\Omega}\frac{1}{H}[(n-1)+nHu_{\nu}]^2\nonumber\\
% &=&\left(\frac{n-1}{n}\right)^2\int_{\partial\Omega}\frac{1}{H}-\frac{n-1}{n}\mbox{Vol}(\Omega)-(n-1)k\int_{\Omega}u.\label{reillyapplied}
%\end{eqnarray}
%Since the left-hand side of \eqref{reillyapplied} is non-negative, we obtain the desired inequality. 
%Furthermore, the equality holds if and only if all integrals in the left-hand side of \eqref{reillyapplied} vanish (since each integral is non-negative). In particular, this implies $\mathring{\nabla}^2 u=0$; therefore, $\Omega$ is  
 
 \end{proof}

Now, let us provide a proof of our Soap Bubble theorem as follows.
 \begin{proof}[Proof of Theorem \ref{2}]
%\begin{theorem} [Soap Bubble Theorem]\label{soapbubble}
%Let $M$ be a manifold such that $Ric\geq (n-1)kg$. Let $u$ be a solution of \eqref{serrinnoover}, where $f'\leq nk$ and $f(0)>0,$ then
%Let $\Omega$ be a domain in a Riemannian manifold $(M^n,g)$ with $\mbox{Ric}\geq (n-1)kg$, for some $k\in\mathbb{R}$, $f'\leq nk$, and $u$ be a solution of \eqref{serrin2}.
%Then 
 %If we define $c$ by \eqref{choiceR} and $H_{0}=\frac{1}{c}$, then
 %\begin{equation*}
%\int_{\partial\Omega}(H_{0}-H)(u_{\nu})^2\geq 0, 
 %\end{equation*}
 %where $H$ is the mean curvature of $\partial\Omega$ and $H_{0}=\frac{(n-1)f(0)}{n c}$ with $c$ a constant given by \eqref{choiceR}. In particular, if $H\geq H_{0}$ on $\partial\Omega$ then $\Omega$ is a ball (and, a fortiori, $u_{\nu}=-c$).
%\end{theorem}

%\begin{proof}

Firstly, note that
\begin{eqnarray*}
\int_{\partial\Omega}Hu_{\nu}^2&=&H_{0}\int_{\partial\Omega}u_{\nu}^2+\int_{\partial\Omega}(H-H_{0})u_{\nu}^2\\
 &=&\frac{(n-1)f(0)}{nc}\int_{\partial\Omega}\left(u_{\nu}+c\right)^2-2\frac{(n-1)f(0)}{n}\int_{\partial\Omega}u_{\nu}-\frac{(n-1)f(0)c}{n}|\partial\Omega|+\int_{\partial\Omega}(H-H_{0})u_{\nu}^2  .
\end{eqnarray*}

  From above equation and \eqref{serrin3} we conclude that
\begin{eqnarray}
& & -\frac{1}{n}\int_{\partial \Omega}u_{\nu}[(n-1)f(0)+nHu_{\nu}]=\nonumber\\
&=&\frac{(n-1)f(0)}{n}\int_{\partial\Omega}u_{\nu}+\frac{(n-1)f(0)c}{n}|\partial\Omega|-\frac{(n-1)f(0)}{nc}\int_{\partial\Omega}\left(u_{\nu}+c\right)^2+\int_{\partial\Omega}(H_{0}-H)u_{\nu}^2\geq 0.\label{reilly3}   
\end{eqnarray}

Now, taking
\begin{equation}
c=-\frac{\int_{\partial\Omega}u_{\nu}}{|\partial\Omega|}.\label{choiceR}
\end{equation}
%$$c=-\frac{\int_{\partial\Omega}u_{\nu}}{|\partial\Omega|}$$\
from \eqref{reilly3} and  \eqref{serrin3} we obtain the desired inequality. Moreover, if $H\geq H_{0}$, the equality holds and, reasoning like the anterior results, we conclude that  $\Omega$ is a metric ball and $u$ is a radial function.

 %\nonumber\\
 %&=&\frac{1}{|\partial\Omega|}\left(Vol(\Omega)+nk\int_{\Omega}u\right)\label{choiceR}
%\end{eqnarray}

\end{proof}

 \section{Some rigidity result via P-function }

Given a solution to the Serrin problem \eqref{serrinnoover}, let us consider the following $P$-function given by 
$$P(x)=|\nabla u|^2(x)+\frac{2}{n}\int_0^{u(x)}f(t)dt,$$
$x\in\Omega.$ Now, we can state and prove the following result.

\begin{lemma}\label{P}
If $u$ is a solution of \eqref{serrinnoover}, $f'\leq nk$ and $Ric\geq (n-1)kg$, then $\Delta P\geq 0$. In particular, if the equality holds $u$ is radial and $\Omega$ is a metric ball.
\end{lemma}

\begin{proof}
Firstly, a straightforward calculation shows that $\nabla P= \nabla^2u(\nabla u)+\frac{2}{n}f(u)\nabla u.$ From the Bochner formula we conclude that 
$$\Delta P=2|\nabla^2 u|^2+2Ric(\nabla u, \nabla u)+2\langle \nabla\Delta u,\nabla u\rangle+\frac{2}{n}f(u)\Delta u+\frac{2}{n}f'(u)|\nabla u|^2.$$

Since we are supposing that $u$ is a solution of \eqref{serrinok1}, $f'\leq nk$ and $Ric\geq (n-1)kg$, we get that 
\begin{eqnarray*}
\Delta P&=&2|\nabla^2 u|^2+2Ric(\nabla u, \nabla u)-\frac{2}{n}(\Delta u)^2+\frac{2-2n}{n}f'(u)|\nabla u|^2\\\nonumber
&= & 2|\mathring{\nabla}^2 u|^2+2Ric(\nabla u, \nabla u)+\frac{2(1-n)}{n}f'(u)|\nabla u|^2\\\nonumber
&\geq & 2|\mathring{\nabla}^2 u|^2\geq 0.
\end{eqnarray*}

On another hand, if $\Delta P=0$ we conclude that the above inequalities are, in fact, equalities. In particular, 
$$(-f'(u)+nk)|\nabla u|^2=0.$$

Thus, we have that $f(u)=f(0)+nku.$ Since $|\mathring{\nabla}^2 u|^2=0$, we can guarantee that $\nabla^2u=-(\frac{f(0)}{n}+ku)g.$ Finally from \cite{FR}, we can conclude that $u$ is radial and $\Omega$ is a metric ball.

\end{proof}

Given $u$ a solution of the problem \eqref{serrinnoover} and let $\Gamma$ be a connected component of $\partial\Omega,$ we define the \emph{generalized normalized wall shear stress} of
$\Gamma$ as 
$$\tau(\Gamma)=\frac{\max_{\partial\Gamma}|\nabla u|^2}{\int_0^{u_{\max}} f(t)dt}.$$

We observe that Espinar and Marín show that the solution of the eigenvalue problem $\Delta u=-2u,$ on the 2-dimensional sphere, is a radial function and the domain $\Omega$ is a geodesic disc, since that $\tau(\Gamma)\leq 1,$ for all connected component of $\partial\Omega,$ see \cite[Theorem 3.1] {espinar}. Following the same ideas, Andrade, Freitas and Marín classify the solutions of the problem \eqref{serrinnoover}, in the particular case $f(u)=-1-nku$, in terms of generalized normalized wall shear stress, see \cite[Lemma 5.1] {andrade}.

%Motivated by this, we have the following extension of their result.

\begin{theorem}\label{torsion}
     Let $M$ be a manifold such that $Ric\geq (n-1)kg$. Let $u$ be a solution of \eqref{serrinnoover}, where $f'\leq nk$. If
     $$\frac{\max_{\partial\Omega}|\nabla u|^2}{\int_0^{u_{\max}} f(t)dt}\leq \frac{2}{n},$$
     where $u_{\max}$ denotes the maximum value of the function $u,$ then $u$ is radial and $\Omega$ is a metric ball.
\end{theorem}

\begin{proof}
 Indeed, since  $P(x)=|\nabla u|^2(x)+\frac{2}{n}\int_0^{u(x)}f(t)dt$, $x\in\Omega,$ and $u=0$ along of the boundary we get that 
 $$\max_{\partial\Omega} P=\max_{\partial\Omega}|\nabla u|^2\leq \frac{2}{n}\int_0^{u_{\max}} f(t)dt.$$

 Taking the point $p_0\in \Omega$ such that $\max_\Omega u=u_{\max}=u(p_0)$ we conclude that $P(p_0)=\frac{2}{n}\int_0^{u_{\max}} f(t)dt.$

Thus, since $P$ is subharmonic function, we conclude by the maximum principle that
$P$ is a constant and, therefore, $u$ is radial and $\Omega$ is a metric ball.

\end{proof}

\begin{proof}[Proof of Theorem \ref{3}]

From Lemma \ref{P}, we get that $\Delta P\geq 0,$ thus from Hopf maximum principle $P_\nu\geq 0$ on the boundary. Now, suppose by contradiction that $P$ is not a constant. Note that 
$P=|\nabla u|^2=\psi^2(r)$ along the boundary and since $\psi$ is increasing  then $\max_{\partial\Omega} P=\psi^2(R).$  

Note that $P_{\nu}=2u_\nu\nabla^2u(\nu,\nu)+\frac{2}{n}f(0)u_\nu$. Thus, from the overdetermined conditions we get that 
$$ P_{\nu}=-2\psi(r)\nabla^2u(\nu,\nu)-\frac{2}{n}f(0)\psi(r).$$
It is not hard to see that 
\begin{eqnarray*}%\label{final2}
 Hu_{\nu}=|\nabla u|\operatorname{div} \frac{\nabla u}{|\nabla u|}=\Delta u-\nabla^2 u(\nu,\nu)=-f(0)-\nabla^2 u(\nu,\nu).
\end{eqnarray*}

 Let $q$ be a point on the boundary such that $P(q)=\max_{\partial\Omega} P=\psi^2(R).$ From above equalities we get at $q,$
 \begin{equation*}
0< P_{\nu}(q)=-2\psi(R)(-f(0)+H(q)\psi(R))-\frac{2}{n}f(0)\psi(R).
\end{equation*}

Thus, 
\begin{equation}\label{comparison}
H(q)<\frac{(n-1)f(0)}{n\psi(R)}.
\end{equation}

\begin{center}
    \begin{tikzpicture}[scale=1.3]

\draw[color=black,thick] (0, 0) circle (1.225);

\draw [fill=black] (0,0) circle (0.6pt);

\draw [fill=black] (-1.17,0.36) circle (0.6pt);

 %\draw [color=black,fill=white,fill opacity=1,thick,domain=56:123.5] plot ({cos(\x)}, {sin(\x)})--(0,0)--(0.553,0.83);

\draw[color=black,line width=0.6pt] (-1,0.6) .. controls (-0.6,1.1)  and (0.2,0.6) .. 
(0.5,0.8) .. controls (1,1.1) and (1.3,0) ..
(0.6,-0.8) .. controls (0.45,-0.99) and (0.4,-1.2)..
(-0.3,-0.8).. controls (-0.55,-0.7) and (-0.45,-0.5)..
(-0.6,-0.3).. controls (-0.75,-0.1) and (-0.85,-0.05) .. 
(-0.9,0) ..controls (-1,0.1) and
(-1.4,0.2).. (-1,0.6) ;

%\draw[line width=0.5pt,color=black,->] (0.7+0.6,0) .. controls (1.3+0.3,0.1) and (1.5+0.3,-0.1) ..  (1.7+0.3,-0.2);

%\draw[line width=0.5pt,color=black,->] (-0.2,-1.23+1.82) .. controls (-0.1,-1.5+1.8) and (0,-1.7+1.9) ..  (0.2,-2+2.1);

 %%%%%%%%%%%%%%%%%%%%%%%%%%%%%%%%%%%%%%

\draw  (0,-0.15) node { \textcolor{black}{$p$}};

\draw  (-1.28,0.36) node { \textcolor{black}{$q$}};

\draw  (0.4,-0.7) node { \textcolor{black}{$\Omega$}};

\draw  (0.6,-1.35) node { \footnotesize\textcolor{black}{$\partial B_R(p)$}};

\end{tikzpicture}
\end{center}

Taking into account that the mean curvature of the geodesic spheres of $M$ as given by $\frac{h'(R)}{h(R)}$, we conclude from the tangency principle that 
$$H(q)\geq \frac{h'(R)}{h(R)}.$$
From \eqref{comparison} and the above inequality, we conclude that $\psi(R)<\frac{(n-1)f(0)}{n}\frac{h(R)}{h'(R)}$ which is a contradiction. Thus, $P$ is constant and, therefore, from Lemma \ref{P} we conclude the desired result.

\end{proof}

\section{A particular case in the annular region}

Firstly, let us consider the  Serrin-type problem in annular domains $\Omega=\Omega_0\setminus\Omega_1,$ 
where  $\Omega_0$ and $\Omega_1$ are bounded domains such that $\Omega_1\subset\Omega_0.$ More precisely, we deal with the following problem
\begin{equation}\label{serrin2}
\left\{\begin{array}{rcl}
\Delta u +nku &=& -n\quad \hbox{in}\quad  \Omega\\
u&=& a \quad \hbox{on}\quad  \Gamma_1\\
u&=& 0\quad \hbox{on}\quad  \Gamma_0,\\
\end{array}\right.
\end{equation}
where $k\in\mathbb{R}$,  $\partial\Omega_0=\Gamma_0,$ $\partial\Omega_1=\Gamma_1$ and $\partial\Omega=\Gamma_1\cup\Gamma_0.$

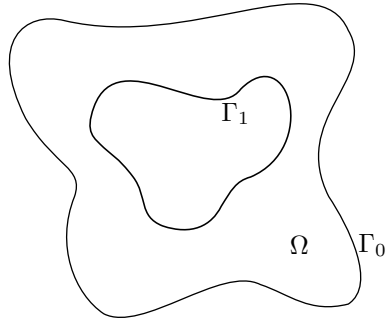
\begin{center}
    \begin{tikzpicture}[scale=1.3]

\draw[line width=0.5pt,color=black] 
(-1.3,1.3) .. controls (0,1) and (1.3,1.9) ..
(1.6,1.2).. controls (1.9,0.7) and (1,0.3) .. 
(1.4,-0.5) .. controls (1.6,-0.8) and (1.9,-1.4) ..  
(1.6,-1.6) .. controls (1.2,-1.7) and (1,-1.5) .. 
(0.7,-1.4) .. controls (0.3,-1.2) and (-0.5,-1.9) ..
(-0.9,-1.7) .. controls (-1.2,-1.5) and (-1.4,-1) ..
(-1.2,-0.5) .. controls (-1.1,-0.2) and (-1.4,-0.2) ..
(-1.7,0.3) .. controls (-2.1,1.1) and (-1.7,1.4) ..  (-1.3,1.3) ;

\draw[color=black,line width=0.6pt,fill=white,fill opacity=1] (-1,0.42) .. controls (-0.75,1.1)  and (0.2,0.2) .. 
(0.5,0.6) .. controls (1,1.1) and (1.3,0) ..
(0.6,-0.3) .. controls (0.45,-0.34) and (0.4,-0.44)..
(0.3,-0.6) .. controls (0.2,-0.8) and (0,-0.9)..
(-0.3,-0.8).. controls (-0.55,-0.7) and (-0.45,-0.5)..
(-0.6,-0.3).. controls (-0.75,-0.1) and (-0.85,-0.05) .. 
(-0.9,0) ..controls (-1,0.1) and
(-1.1,0.15).. (-1,0.42) ;

%\draw[line width=0.5pt,color=black,->] (0.7+0.6,0) .. controls (1.3+0.3,0.1) and (1.5+0.3,-0.1) ..  (1.7+0.3,-0.2);

%\draw[line width=0.5pt,color=black,->] (-0.2,-1.23+1.82) .. controls (-0.1,-1.5+1.8) and (0,-1.7+1.9) ..  (0.2,-2+2.1);

 %%%%%%%%%%%%%%%%%%%%%%%%%%%%%%%%%%%%%%

\draw  (0.45,0.35) node { \textcolor{black}{$\Gamma_1$}};

%\draw  (0.4,0) node { \textcolor{black}{$\Gamma_1$}};

\draw  (1.85,-1) node { \textcolor{black}{$\Gamma_0$}};

%\draw  (2.2,-0.3) node { \textcolor{black}{$\Gamma_0$}};

\draw  (1.1,-1) node { \textcolor{black}{$\Omega$}};

\end{tikzpicture}
\end{center}

\medskip
Again, as an application of the Reilly formula, we have the following integral identity
 
 \begin{theorem}\label{keylemmaanelar}
 Let $u$ be a solution of \eqref{serrin2}, then   
\begin{equation*}%\label{serrin33}  
\int_{\Omega}{| \mathring{\nabla}^2 u}|^2+\int_{\Omega}[\mbox{Ric}-(n-1)kg](\nabla u,\nabla u)=-\int_{\Gamma_0}u_{\nu}[(n-1)+H_0u_{\nu}]-\int_{\Gamma_1}u_{\nu}(H_1u_\nu+(n-1)(ak+1)),
\end{equation*}
where $H_0$ and $H_1$ are the mean curvature of $\Gamma_0$ and $\Gamma_1$ respectively.
 \end{theorem}
 \begin{proof}
 Reilly's identity applied to the solution of \eqref{serrin2} implies
\begin{equation}\label{eq1}
 \int_{\Omega}|\mathring{\nabla}^2 u|^2=-\int_{\partial\Omega}Hu_{\nu}^2-\int_{\Omega}\mbox{Ric}(\nabla u,\nabla u)+\frac{n-1}{n}\int_{\Omega}(\Delta u)^2.   
\end{equation}
 
 Now, we note that from the divergence theorem
 \begin{eqnarray*}
 \int_{\Omega}(n+nku)\Delta u&=&n\int_{\partial \Omega}u_\nu+nk\int_{\Omega}(\frac{\Delta u^2}{2}-|\nabla u|^2)\\
 &=&n\int_{\Gamma_0}u_\nu+\int_{\Gamma_1}(nu_\nu+nka u_\nu)-nk\int_{\Omega}|\nabla u|^2
 \end{eqnarray*}
 
 Plugging the above equality in \eqref{eq1} we obtain the result.
 
 \end{proof}

%\begin{corollary}
%Let $M$ be a manifold with $Ric\geq k(n-1)g$. Suppose that $u$ is a solution of \eqref{serrin2}, such that 
%$u_\nu=-\frac{(n-1)}{H_0}$ and $u_\nu=-\frac{(n-1)(ak+1)}{H_1}$,
%on the boundary components $\Gamma_0$ and $\Gamma_1$ respectively. Then,
%$\Gamma_i$, $i=0,1,$  are umbilical.
%\end{corollary} 

Now, let us consider the following overdetermined problem over bounded annular domains $\Omega$ contained in a manifold $M$, endowed with a conformal vector field $X$, as follows
\begin{equation}\label{serrinconst}
    \begin{cases}
        \Delta u +nku = -n & \hbox{in}\quad  \Omega\\
u= a,\quad u_\nu=c_1  &  \hbox{on} \quad  \Gamma_1\\
u= 0, \quad u_\nu=c_0  & \hbox{on}\quad  \Gamma_0,
    \end{cases}
\end{equation}
where $k,c_0,c_1\in\mathbb{R}$ and $\partial\Omega=\Gamma_0\cup\Gamma_1.$ 

The next result is a Minkowski-type formula addressed to Einstein manifolds endowed with a conformal vector field $X.$

\begin{proposition}\label{keylemma}
Let $M$ be an Einstein manifold with $Ric=(n-1)kg$ and endowed with a closed conformal vector field $X$. Let $u$ be a solution of the Serrin problem \eqref{serrinconst}, then
$$\int_{\Gamma_0}\left(\frac{u_\nu H_0}{(n-1)}+1\right)\langle X,\nu\rangle=-\int_{\Gamma_1}\left(1+\frac{u_\nu H_1}{n-1}+ak\right)\langle X,\nu\rangle,$$
where $H_i$ is the mean curvature of $\Gamma_i.$
\end{proposition}

\begin{proof}
Firstly, since $X$ is a closed conformal vector field there exists a smooth function $\varphi$ such that $\operatorname{div} (X)=n\varphi$ and $\Delta\varphi=-nk\varphi,$ see \cite{KH}.
Taking into account that $\Delta u=-n-nku$ we conclude that 
\begin{equation}\label{main2}
\int_{\Omega}(u\Delta\varphi-\varphi\Delta u)=n\int_{\Omega}\varphi=\int_{\partial\Omega}\langle X,\nu\rangle.
\end{equation}

Now, we also note that from the divergence theorem
$$
\int_{\Omega}\operatorname{div}(u\nabla\varphi-\varphi\nabla u)=-\int_{\Gamma_0}\varphi u_\nu+\int_{\Gamma_1}a\varphi_\nu-\int_{\Gamma_1}\varphi u_\nu
$$
A straightforward calculation shows that $\frac{n-1}{n}\operatorname{div} X=\widetilde{\operatorname{div}}X^T+H_i\langle X,\nu\rangle,$ where $H_i$ denotes the mean curvature of $\Gamma_i.$ Thus, from divergence theorem we get that:
$$
\int_{\Gamma_i}\varphi=\frac{1}{n-1}\int_{\Gamma_i}H_i\langle X,\nu\rangle
$$

From above equations, we deduce that 
\begin{equation}\label{main 1}
\int_{\Omega}\operatorname{div}(u\nabla\varphi-\varphi\nabla u)=-\frac{1}{n-1}\int_{\Gamma_0}H_0 u_\nu \langle X,\nu\rangle+\int_{\Gamma_1}a\varphi_\nu-\frac{1}{n-1}\int_{\Gamma_1}H_1\langle X,\nu\rangle u_\nu.
\end{equation}

On another hand, since $M$ is endowed with a closed vector field $X,$ we have that 
$\nabla_YX=\varphi Y,$ for all $Y\in\mathfrak{X}(M).$
Thus, the curvature tensor
\begin{eqnarray*}
 R(u,v)X%&=&\nabla_{u}\nabla_{v}X-\nabla_{v}\nabla_{u}X-\nabla_{[u,v]}Z\\
 &=&\langle u,\nabla \varphi\rangle v-\langle v,\nabla \varphi\rangle u,
\end{eqnarray*}
for all $u,v\in\mathfrak{X}(M)$. Then $\operatorname{Ric}(X,\cdot)=-(n-1)\nabla\varphi$. 

Thus, since $M$ is Einstein, we have that $-(n-1)\varphi_\nu=\frac{R}{n}\langle X,\nu\rangle$ and, therefore, $\varphi_\nu=-k\langle X,\nu\rangle.$ Finally, from \eqref{main 1} and \eqref{main2}, we obtain the desired result.

\end{proof}

We recall that a domain is called star-shaped with respect to $p$ if each component of the boundary $\partial\Omega$ can be written as a graph over a geodesic sphere with center $p.$ 
Given a domain $\Omega$, endowed with a closed vector field $X$, such that $\partial\Omega=\Gamma_0\cup\Gamma_1$, if $\partial\Omega$ is star-shaped manifold let us assume that $\langle X,\nu\rangle>0$, on $\Gamma_0$, and $\langle X,\nu\rangle <0$, on $\Gamma_1$. 

 \medskip 

\begin{proof}[Proof of Theorem \ref{4}]

%\begin{theorem}\label{main}
%Let $M$ be an Einstein manifold with $Ric=(n-1)kg$ and endowed with a closed conformal vector field $X.$ Suppose that $u$ is a solution of the problem \eqref{serrinconst} such that $c_1 H_1\geq -(ka+1)(n-1)$ on $\Gamma_1$ and $c_0 H_0\leq -(n-1)$ on $\Gamma_0$. If $\partial\Omega$ is star-shaped manifold, then $\partial\Omega$ is an umbilical hypersurface.
%\end{theorem}

%\begin{proof}
Since $\Omega$ has constant scalar curvature,  $\partial\Omega$ is a star-shaped manifold and 
 $u_\nu H_1\geq -(ka+1)(n-1)$ and $u_\nu H_0\leq -(n-1)$, on $\Gamma_1$ and  $\Gamma_0$, respectively, from Proposition \ref{keylemma}, we have that 

$$0\geq \int_{\Gamma_0}\left(\frac{u_\nu H_0}{(n-1)}+1\right)\langle X,\nu\rangle=-\int_{\Gamma_1}\left(1+\frac{u_\nu H_1}{n-1}+ak\right)\langle X,\nu\rangle\geq 0$$

Thus, $(n-1)+H_0u_{\nu}=0$ and $(ka+1)(n-1)+u_\nu H_1=0.$ From, Theorem \ref{keylemmaanelar} we obtain that $\nabla u$ is a closed conformal vector field.
Finally, we conclude that $\Gamma_i$ are umbilical hypersurfaces.

\end{proof}

Taking into account that the space forms can be seen as warped products we have the following result.

\begin{corollary}\label{sphere}
Let $\Omega$ be an annular domain contained in the Euclidean sphere $\mathbb{S}^n$. Suppose that $u$ is a solution of the Serrin problem \eqref{serrinconst} such that $c_1 H_1\geq -(a+1)(n-1)$ on $\Gamma_1$ and $c_0 H_0\leq -(n-1)$ on $\Gamma_0$. If $\partial\Omega$ is star-shaped manifold, then $\Omega$ is the standard annulus $\lbrace x\in\mathbb{S}^n; R< r(x)<R_1\rbrace$, $u$ is radial and given by 
$$u(x)=\frac{1}{\cos(R_1)}(\cos (r(x))-\cos(R_1))$$

\end{corollary}

\begin{proof}
From Theorem \ref{4}, we conclude that components of boundary $\Gamma_0$ and $\Gamma_1$ are umbilical hypersurfaces and, therefore, are geodesic spheres with radius $R$ and $R_1$, respectively. Taking $p$ the center of the geodesic ball $\Gamma_1$ and a normalized unitary geodesic $\gamma$ such that $\gamma(0)=p$ we can introduce the function $f(s)=u(\gamma(s))$. Thus, a straightforward calculation shows that 
$$f''(s)=-f(s)-1,$$
 $f(R)=a$ and $f'(R)= -c_1.$

The general solution of the above ODE is given by $f(s)=x\cos(s)+y\sin(s)-1$. From initial conditions we conclude that 
$x=(a+1)\cos(R)+c_1\sin(R)$ and $y=(a+1)\sin(R)-c_1\cos(R).$

We recall that, $H_1u_\nu+(n-1)(a+1)=0.$ Since the mean curvature is given by $H_1=-(n-1)\frac{\cos(R)}{\sin(R)}$ on $\Gamma_1$, we get that
$$c_1=(a+1)\frac{\sin(R)}{\cos(R)}.$$ Thus, $y=0.$

On the other hand, since $f(R_1)=0$ we conclude that $(a+1)=\frac{\cos(R)}{\cos(R_1)}$ and, therefore, 
$$x=\frac{a+1}{\cos(R)}=\frac{1}{\cos(R_1)}.$$

Then we conclude that $f(s)=\frac{\cos s}{\cos R_1}-1$ and therefore, $u$ is a radial function given by $u(x)=\frac{1}{\cos(R_1)}(\cos (r(x))-\cos(R_1)).$ Taking into account that $u$ is zero along $\Gamma_0$ we conclude that  $\Omega$ is the standard annulus $\lbrace x\in\mathbb{S}^n; R< r(x)<R_1\rbrace$.
\end{proof}

\bigskip

As a consequence of the Theorem \ref{4} we obtain a rigidity result addressed to annular regions contained in the hyperbolic space as follows.
%\begin{theorem}
%Suppose that $u$ is a solution of the Serrin problem such that $c_1H_1\leq -(ka+1)(n-1)$ on $\Gamma_1$ and $c_0H_0\geq -(n-1)$ on $\Gamma_0$. If $\partial\Omega$ is star-shaped manifold, then $\partial\Omega$ is an umbilical hypersurface.
%\end{theorem}

%Thus, we have the following rigidity result for the Serrin problem in domains contained in hyperbolic space as follows.

\begin{corollary}
Let $\Omega$ be an annular domain contained in the hyperbolic space $\mathbb{H}^n.$ Suppose that $u$ is a solution of the Serrin problem such that $c_1 H_1\geq (a-1)(n-1)$ on $\Gamma_1$ and $c_0 H_0\leq -(n-1)$ on $\Gamma_0$. If $\partial\Omega$ is star-shaped manifold, then $\Omega$ is the standard annulus $\lbrace x\in\mathbb{H}^n; R< r(x)<R_1\rbrace$, $u$ is radial and given by 
$$u(x)=\frac{1}{\cosh(R_1)}(-\cosh (r(x))+\cosh(R_1)).$$

\end{corollary}

\vspace{1cm}

\begin{remark}
Given an  annular domain $\Omega$ contained in the sphere, a sufficient condition to guarantee that
 $c_1 H_1\geq -(a+1)(n-1)$ on $\Gamma_1$ and $c_0 H_0\leq -(n-1)$ on $\Gamma_0$ is suppose that  $c_0=-\frac{\sin R_1}{\cos R_1}$, $c_1=\frac{\sin R}{\cos R_1}<0$, $a<-1$
 and suppose that $\Gamma_1$ is a geodesic ball with radius $R.$ Indeed, given an annular domain $\Omega$ contained in the sphere satisfying the above conditions, Lee and Seo, in \cite[Theorem, 2.5]{Keomkyo}, shows that 
the solution of the Serrin problem \eqref{serrinconst}
is a radial function and $\Omega$ is a standard annulus, since that $\partial\Omega$ is star-shaped. In their proof the authors used a $P$-function given by $P(u)=|\nabla u|^2+2u+ku^2$ and shows that $P_\nu\geq 0,$ along of boundary. We point out that

 \begin{equation*}
        \begin{cases}
            P_\nu= -2u_\nu ((a+1)(n-1)+u_\nu H_1) & on \ \Gamma_1\\
             P_\nu= -2u_{\nu}((n-1)+H_0u_{\nu}) & on \ \Gamma_0\\
        \end{cases}
    \end{equation*}
 and, therefore, $c_1 H_1\geq -(a+1)(n-1)$ on $\Gamma_1$ and $c_0 H_0\leq -(n-1),$ on $\Gamma_0.$

%By choosing some constants $c_0$ and $c_1$ and supposing that $\Gamma_1$ is a geodesic sphere, we obtain a new proof of the Theorem 2.4, due to Jihye Lee and Keomkyo Seo, \cite{Keomkyo}

\end{remark}

%\end{proof}

\section*{FUNDING}
The first author acknowledges partial support  from 
	the Coordena\c c\~ao de Aperfei\c coamento de Pessoal de N\'ivel Superior - Brasil (CAPES) - Finance Code 001.
The second author has been partially supported by Conselho Nacional de Desenvolvimento Científico e Tecnológico (CNPq) of the Ministry of Science, Technology and Innovation of Brazil, Grant 306524/2022-8. 

\section*{Acknowledgements}
The authors would like to express their special thanks to Professor Marcos Petrúcio (UFAL) for his valuable contributions to this work.

\bibliographystyle{amsplain}

\end{document}